\documentclass[a4paper,11pt,french,english]{amsart}
\usepackage[english]{babel}
\usepackage[latin1]{inputenc}
\usepackage{lmodern}
\usepackage[T1]{fontenc}
\usepackage{amsmath,amsthm,amssymb,amsfonts}
\usepackage{pstricks}
\usepackage[a4paper,left=3.5cm,right=3.5cm,top=4cm,bottom=4cm]{geometry}

\title{The divergence of the special linear group over a function ring}
\author{Adrien Le Boudec}
\address{Laboratoire de Mathématiques, bâtiment 425, Université Paris-Sud 11, 91405 Orsay, France}
\email{adrien.le-boudec@math.u-psud.fr}
\date{\today}
\subjclass[2010]{Primary 20F65; Secondary 20F69, 15B33}

\theoremstyle{plain}
\newtheorem{thm}{Theorem}[section]
\newtheorem{prop}[thm]{Proposition}
\newtheorem{cor}[thm]{Corollary}
\newtheorem{lem}[thm]{Lemma}

\theoremstyle{definition}
\newtheorem{defi}[thm]{Definition}
\newtheorem{ex}[thm]{Example}

\newtheorem*{claim1}{Claim 1}
\newtheorem*{claim2}{Claim 2}
\newtheorem*{claim3}{Claim 3}
\newtheorem*{claim4}{Claim 4}

\theoremstyle{remark}
\newtheorem{rmq}[thm]{Remark}

\begin{document}

\maketitle

\begin{abstract}
We compute the divergence of the finitely generated group $\mathrm{SL}_n(\mathcal{O}_{\mathcal{S}})$, where $\mathcal{S}$ is a finite set of valuations of a function field, and $\mathcal{O}_{\mathcal{S}}$ is the corresponding ring of \mbox{$\mathcal{S}$-integer} points. As an application, we deduce that all its asymptotic cones are without cut-points.
\end{abstract}

\setcounter{tocdepth}{1}
\tableofcontents

\section{Introduction}

The motivation of this paper comes from the work of Drutu, Mozes and Sapir. In \cite{DMS} they study the geometric notion of divergence in a metric space, which roughly speaking estimates the cost of going from a point $a$ to another point $b$ while remaining outside a large ball centered at a third point $c$. See Section \ref{secdiv} for precise definitions. They particularly focus on the divergence of lattices in higher rank semisimple Lie groups, and prove the following

\begin{thm}[\cite{DMS}, Theorem 1.4] \label{ThmDMS} \enskip \newline
Let $G$ be a semisimple Lie group of $\mathbb{R}$-rank $\geq 2$, and let $\Gamma$ be an irreducible lattice of $G$ which is either of $\mathbb{Q}$-rank one or of the form $\mathrm{SL}_n(\mathcal{O}_{\mathcal{S}})$ with $n \geq 3$, where $\mathcal{S}$ is a finite set of valuations of a number field containing all the Archimedean ones and  $\mathcal{O}_{\mathcal{S}}$ is corresponding ring of $\mathcal{S}$-integers. Then $\Gamma$ has a linear divergence.
\end{thm}

It follows from the work of Stallings \cite{Stallings} that the understanding of the space of ends of the Cayley graph of a finitely generated group yields some significant information about the algebraic structure of the group. Finitely generated groups with no ends are the same as finite groups, and groups with at least two ends are precisely those that split as an HNN-extension or nontrivial amalgam over a finite subgroup. From the geometric viewpoint, being one-ended corresponds to being \textit{connected at infinity}, and the divergence is a quantified version of this connectedness at infinity, estimating how hard it is to connect two given points avoiding a large ball.

The property for a group $\Gamma$ of having a linear divergence is closely related to the existence of cut-points in its asymptotic cones. More precisely, a finitely generated group has a linear divergence if and only if none of its asymptotic cones has cut-points \cite{DMS}. Examples of finitely generated groups with a linear divergence include any direct product of two infinite groups, non-virtually cyclic groups satisfying a law or non-virtually cyclic groups with central elements of infinite order \cite{DS}. It is conjectured in \cite{DMS} that any irreducible lattice in a higher rank semisimple Lie group has a linear divergence. Note that in the case of a cocompact lattice, the conjecture is known to be true because such a lattice is quasi-isometric to the ambient Lie group, so their asymptotic cones (corresponding to the same ultrafilter and scaling constants) are bi-Lipschitz equivalent. Now any asymptotic cone of a semisimple Lie group of $\mathbb{R}$-rank $\geq 2$ is known to have the property that any two points belong to a common flat of dimension $2$ \cite{KL}, and thus does not have cut-points.

In \cite{CDG}, Caprace, Dahmani and Guirardel prove that the divergence of twin building lattices is linear. In particular their result implies that the asymptotic cones of $\mathrm{SL}_n(\mathbb{F}_q[t,t^{-1}])$ do not have cut-points for any $n \geq 2$. Except for this particular case, to the best of our knowledge, nothing has been done before concerning the study of the divergence of arithmetic groups over function fields.

In this paper, we study the divergence of $\mathrm{SL}_n(\mathcal{O}_{\mathcal{S}})$, where $\mathcal{S}$ is a finite non-empty set of pairwise non-equivalent valuations of a global function field,  and $\mathcal{O}_{\mathcal{S}}$ is the corresponding ring of \mbox{$\mathcal{S}$-integers}. By global function field we mean a finite extension of the field $\mathbb{F}_q(t)$ of rational functions with coefficients in the finite field $\mathbb{F}_q$. They are the analogues in positive characteristic of number fields, i.e. finite extensions of the field $\mathbb{Q}$ of rational integers.

\begin{thm} \label{thmmain}
For any $n \geq 3$, the finitely generated group $\mathrm{SL}_n(\mathcal{O}_{\mathcal{S}})$ has a linear divergence.
\end{thm}

Note that Theorem \ref{thmmain} also holds for $n=2$ when $|\mathcal{S}| > 1$ because $\mathrm{SL}_2(\mathcal{O}_{\mathcal{S}})$ quasi-isometrically embeds into a product of $|\mathcal{S}|$ trees, and the image under this embedding can be identified as the complement of disjoint horoballs. Now it follows from Theorem 5.12 of \cite{DMS} that such a space has a linear divergence. On the other hand if $|\mathcal{S}| = 1$ then $\mathrm{SL}_2(\mathcal{O}_{\mathcal{S}})$ is not finitely generated.

Combining Theorem \ref{thmmain} and Proposition \ref{PropDMS} stated in Section \ref{secdiv}, we obtain

\begin{thm} \label{Main}
For any $n \geq 3$, the asymptotic cones of $\mathrm{SL}_n(\mathcal{O}_{\mathcal{S}})$ do not have cut-points.
\end{thm}

The analogy between number fields and function fields extends to groups over these fields, but only up to a certain limit. For instance, an important difference between $\mathrm{SL}_n(\mathbb{Z})$ and $\mathrm{SL}_n(\mathbb{F}_q[t])$ is that the latter fails to be virtually torsion-free, i.e. does not admit a finite index subgroup without torsion elements. Finiteness properties may also change with the characteristic. The group $\mathrm{SL}_2(\mathbb{Z})$ is finitely presented whereas $\mathrm{SL}_2(\mathbb{F}_q[t])$ is not even finitely generated \cite{Nagao}. In the same way, $\mathrm{SL}_3(\mathbb{Z})$ is finitely presented, which is not the case of $\mathrm{SL}_3(\mathbb{F}_q[t])$ \cite{Behr}. Since finite presentability can be interpreted in terms of coarse simple connectedness, it means that $\mathrm{SL}_3(\mathbb{Z})$ and $\mathrm{SL}_3(\mathbb{F}_q[t])$ do not behave the same way from the point of view of simple connectedness at infinity. Our result, in contrast, prove that they behave the same way from the point of view of connectedness at infinity.

\medskip

\noindent \textbf{Organization.} The aim of the next section is to introduce preliminary material concerning asymptotic cones, the notion of divergence and the link with asymptotic cut-points. In Section \ref{secmain} we draw the proof of Theorem \ref{thmmain} in the case when $\mathcal{O}_{\mathcal{S}}$ is the polynomial ring $\mathbb{F}_q[t]$, and the proof in the general case is achieved in Section \ref{secadd}, where we also recall basic facts about number theory.

\medskip

\noindent \textbf{Acknowledgments.} I am very grateful to Yves de Cornulier for various helpful comments and suggestions. I would also like to thank Pierre-Emmanuel Caprace for useful discussions. Part of this work was done during Park City Summer Program on Geometric Group Theory, and its organizers are gratefully acknowledged.

\section{Divergence and asymptotic cut-points} \label{secdiv}

In this section we provide preliminary material about asymptotic cones and the geometric notion of divergence. The main result is the equivalence between the linearity of the growth rate of the divergence and the fact that no asymptotic cone admit cut-points. More details can be found in \cite{DMS}.

\subsection{Asymptotic cones}

We now recall the definition of asymptotic cones. We refer the reader to the survey of Drutu \cite{Drutucone} for more details on asymptotic cones.

Let \mbox{$\omega: \mathcal{P}(\mathbb{N}) \rightarrow \left\{0,1\right\}$} be a non-principal ultrafilter, i.e. a finitely additive probability measure on $\mathbb{N}$ taking values in $\left\{0,1\right\}$ and vanishing on singletons. We say that a statement $P(n)$ holds \mbox{$\omega$-almost} surely if the set of $n$ such that $P(n)$ holds has measure $1$. Given a sequence $(a_n)$ in a topological space $A$, we say that $a \in A$ is an $\omega$-limit of $(a_n)$ if for every neighborhood $U$ of $a$, $a_n \in U$ \mbox{$\omega$-almost} surely. We can check that $\omega$-limits always exist if $A$ is compact and are unique provided that $A$ is Hausdorff.

Consider a non-empty metric space $(X,d)$ and an increasing sequence of real numbers $(s_n)$ such that $s_n \geq 1$ and $\lim\limits_{\omega} s_n = + \infty$. We denote by \mbox{$(X,d_n) = (X,s_n^{-1}d)$} the same space on which we divide the distance by $s_n$. The asymptotic cone is a sort of limit space for the sequence of metric spaces $(X,d_n)$. Given any two sequences $x=(x_n),y=(y_n) \in X^{\mathbb{N}}$, the way $\omega$ forces convergence allows us to define the quantity $d_{\omega}(x,y) = \lim\limits_{\omega} d_n(x_n,y_n),$ with the aim of defining a metric on the product $\prod_n (X,d_n)$. In order to avoid infinite limits, let us fix a base point $e \in X$ and consider \[ \mathrm{Precone}(X,(s_n)) = \left\{(x_n) \in X^{\mathbb{N}} \, : \, d_n(x_n,e) \, \text{is bounded} \right\}. \] This does not depend on the choice of $e \in X$. The function $d_{\omega}$ is a pseudo-metric on $\mathrm{Precone}(X,(s_n))$, i.e. $d_{\omega}$ satisfies the triangle inequality, is symmetric and vanishes on the diagonal, but it need not be a metric because the $\omega$-limit of $d_n(x_n,y_n)$ can be zero for two different sequences, for example for two sequences equal \mbox{$\omega$-almost} surely. The asymptotic cone $\mathrm{Cone}_{\omega}(X,(s_n))$ of $(X,d)$ relative to the sequence of scaling constants $(s_n)$, the base point $e \in X$ and the non-principal ultrafilter $\omega$, is defined by identifying elements of $\mathrm{Precone}(X,(s_n))$ at distance zero: \[ \mathrm{Cone}_{\omega}(X,(s_n)) = \mathrm{Precone}(X,(s_n)) / \! \! \sim \, ,\] where $(x_n) \sim (y_n)$ if $d_{\omega}(x,y) = 0$.

In particular, the asymptotic cones of a finitely generated group $\Gamma$ are the asymptotic cones of its Cayley graph endowed with some word metric. Note that in this context, $\mathrm{Precone}(\Gamma,(s_n))$ is endowed with a group structure inherited from $\Gamma$, and acts by isometries on $\mathrm{Cone}_{\omega}(\Gamma,(s_n))$ by $(g_n) \cdot [(h_n)] = [(g_nh_n)]$. This action being transitive, $\mathrm{Cone}_{\omega}(\Gamma,(s_n))$ is a homogeneous metric space.

\subsection{The divergence function}

As usual in the context of studying the large-scale geometry of finitely generated groups, we consider the following equivalence relation on the set of functions measuring asymptotic properties of groups.

\begin{defi}
Let $f,g: \mathbb{R}_+ \rightarrow \mathbb{R}_+$. We write $f \preccurlyeq g$ if there exists $C > 0$ such that for all $x \in \mathbb{R}_+$, \[f(x) \leq C g(C x + C) + C x + C. \] If $f,g: \mathbb{R}_+ \rightarrow \mathbb{R}_+$ satisfy $f \preccurlyeq g$ and $g \preccurlyeq f$, then we write $f \simeq g$, and $f,g$ are said to be $\simeq$-equivalent.
\end{defi}

In \cite{DMS} the authors introduce several definitions of divergence, and prove that they give $\simeq$-equivalent functions. We will focus on the following definition of divergence.

\begin{defi}
Let $(X,d)$ be a geodesic metric space, and let $0 < \delta < 1$ and $\gamma \geq 0$. We define the divergence $\mathrm{div}_\gamma(a,b,c;\delta)$ of a pair of points $a,b \in X$ relative to a point $c \in X$, to be the length of a shortest path in $X$ connecting $a$ and $b$ and avoiding the ball centered at $c$ of radius $\delta d(c,\left\{a,b\right\}) - \gamma$. If there is no such path, put $\mathrm{div}_\gamma(a,b,c;\delta) = \infty$. The divergence $\mathrm{div}_\gamma(a,b;\delta)$ of the pair $(a,b)$ is defined as the supremum of the divergences relative to all points $c \in X$, and the divergence function $\mathrm{Div}_\gamma(n;\delta)$ is defined as the supremum of all divergences of pairs $(a,b)$ with $d(a,b) \leq n$.
\end{defi}

The following proposition is proved in Lemma 3.4 and Lemma 3.11 of \cite{DMS}.

\begin{prop}
Let $X$ be a connected homogeneous locally finite graph with one end (e.g. the Cayley graph of a finitely generated one-ended group). Then there exist $\gamma_0, \delta_0 > 0$ such that for every $\gamma \geq \gamma_0$ and every $\delta \leq \delta_0$, the function $n \mapsto \mathrm{Div}_\gamma(n;\delta)$ takes only finite values, and is independent, up to the equivalence relation $\simeq$, of the parameters $\gamma$ and $\delta$.
\end{prop}

As usual when studying the large-scale geometry of metric spaces, we want to work with objects that behave well with respect to quasi-isometries. Recall that a map $f: (X,d_X) \rightarrow (Y,d_Y)$ is a quasi-isometric embedding if for some constants $L \geq 1$ and $C \geq 0$, \[L^{-1}d_X(x_1,x_2) - C \leq d_Y(f(x_1),f(x_2)) \leq L d_X(x_1,x_2) + C \] for all $x_1,x_2 \in X$. If moreover $Y$ is contained in the \mbox{$C$-neighborhood} of $f(X)$ then $f$ is called a quasi-isometry, and $X,Y$ are said to be quasi-isometric. The following result will be used without further mention, see Lemma 3.2 in \cite{DMS}.

\begin{prop}
Let $X$ and $Y$ be connected homogeneous locally finite graphs with one end. If $X$ and $Y$ are quasi-isometric then they have $\simeq$-equivalent divergence functions. 
\end{prop}

Let $X$ be an infinite connected homogeneous graph. As we claimed in the introduction, the topological property for asymptotic cones of having cut-points is closely related to the geometric property for $X$ of having a linear divergence, i.e. a divergence function $\simeq n$. Recall that in a geodesic metric space $E$, a cut-point is a point $p \in E$ such that $E \setminus \left\{p\right\}$ has more than one path-connected component. For a proof of the following result, see Lemma 3.17 in \cite{DMS}.

\begin{prop} \label{PropDMS}
All the asymptotic cones of $X$ have no cut-points if and only if there exist $\gamma, \delta$ such that the function $\mathrm{Div}_\gamma(n,\delta)$ is linear.
\end{prop}

\section{The divergence of $ \mathrm{SL}_n(\mathbb{F}_q[t])$} \label{secmain}

In this section we prove that the finitely generated group $\mathrm{SL}_n(\mathbb{F}_q[t])$ has a linear divergence, explicitly constructing a short path joining any two given large matrices and avoiding a large ball around the origin. The Cayley graph of $\mathrm{SL}_n(\mathbb{F}_q[t])$ associated to a finite generating set $S$ is the graph with $\mathrm{SL}_n(\mathbb{F}_q[t])$ as set of vertices, in which $(\gamma_1,\gamma_2)$ is an edge if and only if there exists $s \in S$ such that $\gamma_2=\gamma_1 s$. It follows that moving in the Cayley graph from a vertex to another corresponds to making column operations in terms of matrices. We will write down the proof only for $n=3$ but our method applies directly to the general case.

\subsection{Preliminary material}

Throughout this section, we let $\Bbbk$ be the the field $\mathbb{F}_q(t)$ and $\Bbbk_{\infty}$ be the field of Laurent series $\mathbb{F}_q(\!(t^{-1})\!)$. Recall that $\Bbbk_{\infty}$ is the completion of $\Bbbk$ with respect to the valuation $\nu_{\infty}$, see Example \ref{exval}. Let us denote by $|\cdot|$ the associated norm. Note that for any polynomial $a \in \mathbb{F}_q[t]$, we have $|a|=q^{\deg a}$. If $\gamma$ is a matrix with entries in $\mathbb{F}_q[t]$, we denote by \[ \left\| \gamma \right\| = \max_{i,j} \left\{ | \gamma_{i,j} | \right\}\] its norm as a matrix over $\Bbbk_{\infty}$.

We now recall some basic results coming from the theory of matrices with elements in a euclidean ring (see Theorem 22.4 of \cite{McD} for example). We let $e_{i,j}(r)$ be the elementary unipotent matrix whose $(i,j)$-entry is $r$, $i \neq j$.

\begin{thm} \label{MCD}
Let $A$ be a euclidean ring and $n \geq 2$. Then any element of $\mathrm{SL}_n(A)$ is a product of elementary matrices. \qed
\end{thm}

\begin{cor} \label{corMCD}
For any $n \geq 3$, the group $\mathrm{SL}_n(\mathbb{F}_q[t])$ is generated by the finite set \[ S_0 = \bigcup_{i\neq j} \left\{ e_{i,j}(\alpha) \, : \, \alpha \in \mathbb{F}_q^* \right\} \cup \left\{ e_{i,j}(t) \right\}. \]
\end{cor}

\begin{proof}
According to Theorem \ref{MCD} it is enough to prove that for any $P \in \mathbb{F}_q[t]$, the matrix $e_{i,j}(P)$ is a product of elements of $S_0$. Since $\mathbb{F}_q^*$ and $t$ generate $\mathbb{F}_q[t]$ as a ring, this follows, using a straightforward induction, from the identity \[ e_{i,j}(x+y)=e_{i,j}(x)e_{i,j}(y) \] and the commutator relation \[ e_{i,j}(xy)=[e_{i,k}(x),e_{k,j}(y)] \] for all $x,y \in \mathbb{F}_q[t]$ and $k \neq i,j$ (using that $n\ge 3$).
\end{proof}

Since all the Cayley graphs of $\mathrm{SL}_3(\mathbb{F}_q[t])$ are quasi-isometric, we can choose a particular generating set $S$. Its construction goes as follows.

Let us still denote by $S_0$ the previous finite generating set. We will add some elements to $S_0$ in order to get a more convenient one.

Denote by $S_1$ the set of monomial matrices whose non-zero entries are $\pm 1$ (by a monomial matrix we mean a matrix with exactly one non-zero element in each row and each column).

Consider a matrix $A_1 \in \mathrm{SL}_2(\mathbb{F}_q[t])$ with two eigenvalues $\lambda_+, \lambda_- \in \Bbbk_{\infty}$ such that $\left|\lambda_+\right| > 1$ and $\left| \lambda_- \right| < 1$, and another matrix $A_2$ with the same eigenvalues but with different eigen-directions. For example the matrices \[ A_1 = \begin{pmatrix} 1 & t \\ 1 & t+1 \end{pmatrix} \] and its conjugate \[ A_2 = \begin{pmatrix} 1 & t \\ 0 & 1 \end{pmatrix} \begin{pmatrix} 1 & t \\ 1 & t+1 \end{pmatrix} \begin{pmatrix} 1 & t \\ 0 & 1 \end{pmatrix}^{-1} = \begin{pmatrix} t+1 & t \\ 1 & 1 \end{pmatrix}\] would work. To complete our generating set, put \[ S_2 = \bigcup_{i=1}^2 \left\{ \begin{pmatrix}
 A_i & \begin{array}{l} 0 \\ 0 \end{array} \\
 \begin{array}{cc} 0 & 0 \end{array} & 1
\end{pmatrix},
\begin{pmatrix}
1 & \begin{array}{ll} 0 & 0 \end{array} \\
\begin{array}{l} 0 \\ 0 \end{array} &  A_i
\end{pmatrix}
\right\},
\]
and take $S = S_0 \cup S_1 \cup S_2$.

We now set some notation and terminology, extensively borrowed from \cite{DMS}.

\begin{defi}
Let $0 < \varepsilon < 1$. An entry $\gamma_{i,j}$ of $\gamma \in \mathrm{SL}_3(\mathbb{F}_q[t])$ is called $\varepsilon$-large if we have \[ \log \left| \gamma_{i,j} \right| \geq \varepsilon \log \left\| \gamma \right\|. \]
\end{defi}

Note that any entry $\gamma_{i,j}$ of $\gamma \in \mathrm{SL}_3(\mathbb{F}_q[t])$ such that $\left| \gamma_{i,j} \right| = \left\| \gamma \right\|$ is $\varepsilon$-large, but these may not be the only $\varepsilon$-large entries.

\begin{defi}
Let $\delta_1, \delta_2, \delta_3 > 0$. A $(\delta_1, \delta_2, \delta_3)$-external trajectory connecting two elements $\gamma_1, \gamma_2$ of $\mathrm{SL}_3(\mathbb{F}_q[t])$ is a path $\eta$ in its Cayley graph from $\gamma_1$ to $\gamma_2$ such that:
\begin{enumerate}
	\item [i)] $\eta$ remains outside the ball centered at $e$ of radius \mbox{$\delta_1 \times d_{S}(e,\left\{\gamma_1, \gamma_2\right\}) - \delta_2 $};
	\item [ii)] the length of $\eta$ is bounded by $\delta_3 \times d_{S}(\gamma_1,\gamma_2)$.
\end{enumerate}
Two elements $\gamma_1$ and $\gamma_2$ are said to be $(\delta_1, \delta_2, \delta_3)$-externally connected if there exists a $(\delta_1, \delta_2, \delta_3)$-external trajectory between them, and uniformly externally connected if there exist some constants $\delta_1, \delta_2, \delta_3 > 0$ which do not depend on $\gamma_1$ and $\gamma_2$, such that $\gamma_1$ and $\gamma_2$ are $(\delta_1, \delta_2, \delta_3)$-externally connected.
\end{defi}

Clearly $\mathrm{SL}_3(\mathbb{F}_q[t])$ has a linear divergence if and only if any two elements are uniformly externally connected.

In the next section, when constructing paths in the Cayley graph of $\mathrm{SL}_3(\mathbb{F}_q[t])$, we will substantially make use of the work of Lubotzky, Mozes and Raghunathan. In \cite{LMR} they prove that any irreducible lattice in a semisimple Lie group of $\mathbb{R}$-rank $\geq 2$, endowed with some word metric associated to a finite generating set, is quasi-isometrically embedded in the ambient Lie group. Their result allows us to estimate the size of a matrix $\gamma \in \mathrm{SL}_3(\mathbb{F}_q[t])$ with respect to the word metric $d_S$ in terms of its norm $\left\| \gamma \right\|$.

\begin{prop} \label{propLMR}
There exists $C > 0$ such that for any $\gamma \in \mathrm{SL}_3(\mathbb{F}_q[t])$, \begin{equation} \label{eq1} C^{-1}(1 + \log \left\| \gamma \right\|) \leq d_{S}(e, \gamma) \leq C(1 + \log \left\| \gamma \right\|). \end{equation}
\end{prop}

Roughly speaking, Proposition \ref{propLMR} states that the size of \mbox{$\gamma \in \mathrm{SL}_3(\mathbb{F}_q[t])$} is the maximum of the degrees of the entries of $\gamma$.

\subsection{Quasi-isometrically embedded lamplighters}

While connecting points in the Cayley graph of $\mathrm{SL}_3(\mathbb{F}_q[t])$, we will take advantage of embedded subgroups with nice geometric properties with respect to our problem, i.e. with a linear divergence. For $i=1,2$, define
\[ \Lambda_i = \begin{pmatrix}
1 & \begin{array}{ll} 0 & 0 \end{array} \\
\begin{array}{l} \mathbb{F}_q[t] \\ \mathbb{F}_q[t] \end{array} &  A_i^{\mathbb{Z}}
\end{pmatrix}. \]
Since $\mathbb{F}_q[t]$ is a cocompact lattice in $\Bbbk_{\infty}$, the group $\Lambda_i$ is a cocompact lattice in \[ G_i = \begin{pmatrix}
1 & \begin{array}{ll} 0 & 0 \end{array} \\
\begin{array}{l} \Bbbk_{\infty} \\ \Bbbk_{\infty} \end{array} &  A_i^{\mathbb{Z}}
\end{pmatrix}. \]
Now since $A_i$ is diagonalizable in $\mathrm{SL}_2(\Bbbk_{\infty})$, it follows that $G_i$ is conjugated in $\mathrm{SL}_3(\Bbbk_{\infty})$ to the group \[ \begin{pmatrix} 1 & 0 & 0 \\ \Bbbk_{\infty} & \lambda_+^n & 0 \\ \Bbbk_{\infty} & 0 & \lambda_+^{-n} \end{pmatrix}, \]
with $\left| \lambda_+ \right| > 1$. So the groups $\Lambda_i$ both are quasi-isometric to $\Bbbk_{\infty}^2 \rtimes \mathbb{Z}$, where the action of $\mathbb{Z}$ is the multiplication by $(\lambda_+,\lambda_+^{-1})$, and in particular we obtain

\begin{prop}
For $i=1,2$, all the asymptotic cones of $\Lambda_i$ are homeomorphic to the hypersurface of equation $b(x)+b(y)=0$ in the product of two copies of the universal real tree $\mathbb{T}$ with continuum branching everywhere, where $b$ is a Busemann function on $\mathbb{T}$.
\end{prop}

\begin{proof}
See Section 9 of \cite{Co}.
\end{proof}

Since this space has no cut-points, Proposition \ref{PropDMS} tells us that the groups $\Lambda_i$ have a linear divergence. In order to use this property inside $\mathrm{SL}_3(\mathbb{F}_q[t])$, we need to ensure that these groups are quasi-isometrically embedded. This follows from the fact that the groups $\Lambda_i$ are quasi-isometric to $\Bbbk_{\infty}^2 \rtimes \mathbb{Z}$ and from the following standard lemma, which gives an estimate of the word metric in the latter group.

\begin{lem}
Let $(\mathbb{K}_i)_{i=1..m}$ be a family of local fields, each one being endowed with a multiplicative norm $| \cdot |_i$. For $i=1\ldots m$, let $\lambda_i \in \mathbb{K}_i$ be an element of norm different from 1. Let $G = (\bigoplus \mathbb{K}_i) \rtimes \mathbb{Z}$, where the diagonal action of $\mathbb{Z}$ is by multiplication by $\lambda_i$ on $\mathbb{K}_i$. Then $G$ is compactly generated and for any compact generating set $S$ of $G$, there exists a constant $C$ such that for any $(x,n) = (x_1, \ldots, x_m,n) \in G$, \[ C^{-1} \left( \log \left(1+ \max_i |x_i|_i \right) + |n| \right) \leq |(x,n)|_S \leq C \left( \log \left(1+ \max_i |x_i|_i \right) + |n| \right). \]
\end{lem}

\begin{proof}
It is not hard to see that it is enough to prove the result for only one local field $\mathbb{K}$, and that without loss of generality we can assume that the element $\lambda$ defining the action of $\mathbb{Z}$ has norm strictly smaller than $1$. In this situation $G$ is generated by the compact set $S = \mathbb{K}_0 \cup t$, where $\mathbb{K}_0$ denotes the set of elements of $\mathbb{K}$ of norm at most $1$, and $t=(0,1)$. Since different compact generating sets yield bi-Lipschitz equivalent word metrics, it is enough to prove the result for this generating set. 

Let us prove the upper bound first. If $(x,n) \in G$ and if we denote by \[ n_0 = \max \left( \left\lfloor  - \log |x| / \log |\lambda| \right\rfloor + 1, 0 \right), \] then the reader can check that $\lambda^{n_0}x \in \mathbb{K}_0$. Therefore there exists $x_0 \in \mathbb{K}_0$ such that the word $t^{-n_0}x_0t^{n_0}t^n$ represents the element $(x,n)$ of $G$, which therefore has length at most $2n_0+1+|n| \leq c ( \log (1+|x|) + |n|)$ for some constant $c$.

To prove the lower bound, let us consider a word $t^{n_1}x_1 \ldots t^{n_k}x_k$ of length at most $\ell$ representing an element $(x,n)$ of $G$, where $x_i \in \mathbb{K}_0$ and $k, \sum |n_i| \leq \ell$. Then $(x,n)$ is also represented by the word \[ t^{n_1}x_1t^{-n_1} t^{n_1+n_2}x_2t^{-n_1-n_2} \ldots t^{n_1 + \ldots + n_k}x_k t^{- n_1 - \ldots - n_k} t^{n_1 + \ldots + n_k}, \] and therefore the equality \[ \sum_{i=1}^k \lambda^{n_1+\ldots+n_i}x_i = x \] holds in $\mathbb{K}$, and $n = n_1 + \ldots + n_k$. Consequently \[ |x| \leq \sum_{i=1}^k |\lambda^{n_1+\ldots+n_i}x_i| \leq k \max_i |\lambda|^{n_1+\ldots+n_i} \leq \ell |\lambda|^{-\ell} \leq |\lambda|^{-c'\ell}\] for some constant $c'$. So now \[ \log (1+|x|) + |n| \leq \log (1+ |\lambda|^{-c'\ell}) + \sum |n_i| \leq c'' \ell, \] for some constant $c''$ and $\ell$ large enough, which completes the proof.
\end{proof}

\subsection{Construction of external trajectories}

The following lemma will be used repeatedly in this section.

\begin{lem} \label{tool}
Let $0 < \varepsilon < 1$ be fixed. For any $\alpha \in \mathrm{SL}_3(\mathbb{F}_q[t])$ having an $\varepsilon$-large entry in its third column and any \[ \theta = \begin{pmatrix} 1 & 0 & 0 \\ x & 1 & 0 \\ y & 0 & 1 \end{pmatrix} \in \mathrm{SL}_3(\mathbb{F}_q[t]),\]
$\alpha$ and $\alpha \theta$ are uniformly externally connected.
\end{lem}

\begin{proof}
Let $\alpha \in \mathrm{SL}_3(\mathbb{F}_q[t])$ having an $\varepsilon$-large entry $\alpha_{j,3}$ in its third column. First note that we have a control of the size of $\alpha \theta$ in the sense that $\alpha \theta$ is bounded away from the identity. Indeed, since the third column of $\alpha \theta$ is the third column of $\alpha$,
\[ \log \left\|\alpha \theta\right\| \geq \log \left|\alpha_{j,3} \right| \geq \varepsilon \log \left\|\alpha \right\|. \]
Now according to (\ref{eq1}), \[ \begin{aligned} d_S(e,\alpha \theta) & \geq C^{-1}(1+\log \left\|\alpha \theta\right\|) \\ & \geq C^{-1}(1+\varepsilon \log \left\|\alpha \right\|) \\ & \geq C^{-1}(1+\varepsilon (C^{-1}d_S(e,\alpha) - 1)) \\ & = \varepsilon C^{-2} d_S(e,\alpha) + C^{-1}(1-\varepsilon). \end{aligned} \]

Since $\Lambda_i$ is quasi-isometrically embedded in $\mathrm{SL}_3(\mathbb{F}_q[t])$, for every \[ \theta = \begin{pmatrix} 1 & 0 & 0 \\ x & 1 & 0 \\ y & 0 & 1 \end{pmatrix} \in \mathrm{SL}_3(\mathbb{F}_q[t])\] there exists a short word $s_1^{(i)} \ldots s_n^{(i)}$ representing $\theta$, where all the $s_k^{(i)}$ belong to \[ \left\{
\begin{pmatrix}
1 & \begin{array}{ll} 0 & 0 \end{array} \\
\begin{array}{l} 0 \\ 0 \end{array} &  A_i
\end{pmatrix},
\begin{pmatrix} 1 & 0 & 0 \\ \alpha & 1 & 0 \\ 0 & 0 & 1 \end{pmatrix},
\begin{pmatrix} 1 & 0 & 0 \\ 0 & 1 & 0 \\ \beta & 0 & 1 \end{pmatrix}
\right\},
\]
with $\alpha, \beta \in \mathbb{F}_q^*$. It gives us two short trajectories $\eta_1, \eta_2$ from $\alpha$ to $\alpha \theta$. To conclude that $\alpha$ and $\alpha \theta$ are uniformly externally connected, it is enough to prove that one of these two trajectories does not get too close to the identity.

Following \cite{DMS}, for $0<c<1$ and $i=1,2$, we define the non-contracting cone of the matrix $A_i$ in $\Bbbk_{\infty}^2$, \[ \textnormal{NC}_c(A_i) = \left\{ v \in \Bbbk_{\infty}^2 \, : \, \left\|v A \right\| > c \, \left\| v \right\| \, \, \forall A \in \left\langle A_i\right\rangle \right\} \cup \left\{0\right\}. \] Since $A_1$ and $A_2$ have distinct eigen-directions, it is easily checked that we can choose $c_0 > 0$ small enough such that \[ \textnormal{NC}_{c_0}(A_1) \cup \textnormal{NC}_{c_0}(A_2) = \Bbbk_{\infty}^2,\] and therefore we can find $i$ such that $(\alpha_{j,2},\alpha_{j,3}) \in \textnormal{NC}_{c_0}(A_i)$, where $\alpha_{j,3}$ still denotes an $\varepsilon$-large entry of $\alpha$.

Note that in each trajectory $\eta_i$, only the elements $s_k^{(i)}$ belonging to $S_2$ can affect the second or the third column. It follows from this observation, and from the choice of $i$ such that $(\alpha_{j,2},\alpha_{j,3}) \in \textnormal{NC}_{c_0}(A_i)$, that the same computation we made at the beginning of the proof yields that any point in the trajectory $\eta_i$ is at distance at least \[ \varepsilon C^{-2} d_S(e,\alpha) + C^{-1}(1-\varepsilon + \log c_0) \] from the identity, which concludes the proof.
\end{proof}

As in the proof in \cite{DMS} for the case of $\mathrm{SL}_3(\mathbb{Z})$, the strategy for uniformly externally connecting two given points will consist in first moving each of them to another point of the same size, lying in a particular subgroup. This is achieved by the following lemma.

\begin{lem} \label{lem1}
There exist $c_1, c_2 > 0$ such that for any $\alpha \in \mathrm{SL}_3(\mathbb{F}_q[t])$, there is an element \[ \alpha' = \begin{pmatrix} 1 & 0 & x \\ 0 & 1 & y \\ 0 & 0 & 1 \end{pmatrix} \in \mathrm{SL}_3(\mathbb{F}_q[t]) \] satisfying \[ c_1 \leq \frac{1+d_{S}(e,\alpha)}{1+d_{S}(e,\alpha')} \leq c_2, \] and such that $\alpha$ and $\alpha'$ are uniformly externally connected. 
\end{lem}

Then using the fact that the subgroup
\[ \Lambda_1' = \begin{pmatrix}
 A_1 & \begin{array}{l} x \\ y \end{array} \\
 \begin{array}{cc} 0 & 0 \end{array} & 1
\end{pmatrix}\]
(which is isomorphic to $\Lambda_1$) has a linear divergence and is quasi-isometrically embedded inside $\mathrm{SL}_3(\mathbb{F}_q[t])$, we obtain that any two points in $\mathrm{SL}_3(\mathbb{F}_q[t])$ are uniformly externally connected, i.e. $\mathrm{SL}_3(\mathbb{F}_q[t])$ has a linear divergence.

\begin{proof}[Proof of Lemma \ref{lem1}]
We are going to proceed in a finite number of steps to construct a path in the Cayley graph connecting $\alpha$ to a suitable $\alpha'$. In each step, we can check that $\alpha_{i+1}$ and $\alpha_i$ satisfy \[ c_1^{(i)} \leq \frac{1+d_{S}(e,\alpha_i)}{1+d_{S}(e,\alpha_{i+1})} \leq c_2^{(i)}\] for some constants $c_1^{(i)},c_2^{(i)} > 0$ which do not depend on $\alpha$. The estimate on the size of $\alpha'$ follows from this sequence of inequalities. 

Let $0 < \varepsilon < 1$ and let 
\[\alpha =
\begin{pmatrix}
 * & * & * \\
 * & * & * \\
 a & b & c
\end{pmatrix} \in \mathrm{SL}_3(\mathbb{F}_q[t]).\]
Swapping columns if necessary, which is possible thanks to $S_1$, we can assume that $\alpha$ has an $\varepsilon$-large entry $\alpha_{i,3}$ in the last column.

Now we claim that without loss of generality, we can also assume that the lower right entry $c$ is $\varepsilon$-large. Indeed, we can first assume, at the price of exchanging the first two columns, that $b \neq 0$ (otherwise if $a=b=0$ then we can directly move to Claim 4 of this proof). Now multiplying $\alpha$ on the right by \[ \theta = \begin{pmatrix} 1 & 0 & 0 \\ t^{\deg \alpha_{i,3}} & 1 & 0 \\ 0 & 0 & 1 \end{pmatrix}\] has the effect of changing $a$ into $a+  t^{\deg \alpha_{i,3}}b$, which has degree at least $\deg \alpha_{i,3}$ because $b \neq 0$. But according to Lemma \ref{tool}, $\alpha$ and $\alpha \theta$ are uniformly externally connected. Therefore we can suppose that the lower left entry is $\varepsilon$-large, and finally we just have to exchange the first and third columns to obtain an $\varepsilon$-large lower right entry.

\begin{claim1}
$\alpha$ is uniformly externally connected to \[ \alpha_2 = \begin{pmatrix}
 * & * & * \\
 * & * & * \\
 a' & b' & c
\end{pmatrix},
\]
where $a',b'$ verify $\gcd(a',b')=1$.
\end{claim1}

\begin{proof}
Multiplying $\alpha$ by $e_{1,3}(1)$ if necessary, which has the effect of changing $a$ into $a' = a+c$, we can assume that $a' \neq 0$. Let $p_1, \ldots, p_d$ be the distinct prime divisors of $a'$. If we let $m$ be the product of all the $p_i$'s dividing neither $b$ nor $c$, then it is easy to check that $a'$ and $b'=b+mc$ are relatively prime. Now using Lemma \ref{tool} and the identity
\[
\begin{pmatrix}
 1 & 0 & 0 \\
 0 & 1 & 0 \\
 0 & m & 1
\end{pmatrix}
=
\begin{pmatrix}
 0 & 1 & 0 \\
 -1 & 0 & 0 \\
 0 & 0 & 1
\end{pmatrix}
\begin{pmatrix}
 1 & 0 & 0 \\
 0 & 1 & 0 \\
 -m & 0 & 1
\end{pmatrix}
\begin{pmatrix}
 0 & -1 & 0 \\
 1 & 0 & 0 \\
 0 & 0 & 1
\end{pmatrix},
\]
we obtain that $\alpha$ is uniformly externally connected to \[ \begin{pmatrix}
 * & * & * \\
 * & * & * \\
 a' & b' & c
\end{pmatrix} = \alpha_2.
\]
\end{proof}

We now want to perform an external trajectory between $\alpha_2$ and \[ \alpha_3 = \begin{pmatrix}
 * & * & * \\
 * & * & * \\
 a' & b'' & c
\end{pmatrix},
\]
where the entry $b''$ is $\varepsilon$-large and satisfies $\gcd(a',b'')=1$. There is nothing to do if $\deg b' \geq \deg c$. Otherwise we get this external trajectory by setting $b'' = b' + t^{\deg c}a'$ and using Lemma \ref{tool} and the identity
\[
\begin{pmatrix}
 1 & t^{\deg c} & 0 \\
 0 & 1 & 0 \\
 0 & 0 & 1
\end{pmatrix}
=
\begin{pmatrix}
 0 & -1 & 0 \\
 1 & 0 & 0 \\
 0 & 0 & 1
\end{pmatrix}
\begin{pmatrix}
 1 & 0 & 0 \\
 -t^{\deg c} & 1 & 0 \\
 0 & 0 & 1
\end{pmatrix}
\begin{pmatrix}
 0 & 1 & 0 \\
 -1 & 0 & 0 \\
 0 & 0 & 1
\end{pmatrix}.
\]

\begin{claim2}
$\alpha_3$ is uniformly externally connected to \[ \alpha_4 = \begin{pmatrix}
 * & * & * \\
 * & * & * \\
 a' & 1 & b''
\end{pmatrix}.
\]
\end{claim2}

\begin{proof}
First note that we can go in one step from $\alpha_3$ to \[ \alpha_3' = \alpha_3 \begin{pmatrix}
 1 & 0 & 0 \\
 0 & 0 & 1 \\
 0 & -1 & 0
\end{pmatrix}
=
\begin{pmatrix}
 * & * & * \\
 * & * & * \\
 a' & -c & b''
\end{pmatrix}.
\]
Let $u,v$ be Bézout coefficients of small degree for the pair $(a',b'')$, that is, $a'u+b''v=1$. Using Lemma \ref{tool} one more time and the two identities already used above, we get that $\alpha_3'$ is uniformly externally connected to
\[ \alpha_3' \begin{pmatrix}
 1 & (c+1)u & 0 \\
 0 & 1 & 0 \\
 0 & 0 & 1
\end{pmatrix}
\begin{pmatrix}
 1 & 0 & 0 \\
 0 & 1 & 0 \\
 0 & (c+1)v & 1
\end{pmatrix}
= \alpha_3'
\begin{pmatrix}
 1 & (c+1)u & 0 \\
 0 & 1 & 0 \\
 0 & (c+1)v & 1
\end{pmatrix}
=
\begin{pmatrix}
 * & * & * \\
 * & * & * \\
 a' & 1 & b''
\end{pmatrix}.
\]
\end{proof}

\begin{claim3}
$\alpha_4$ is uniformly externally connected to \[ \alpha_5 = \begin{pmatrix}
 * & * & * \\
 * & * & * \\
 0 & 1 & 1
\end{pmatrix},
\]
where the third column of $\alpha_5$ remains $\varepsilon$-large.
\end{claim3}

\begin{proof}
It directly follows from Lemma \ref{tool} that $\alpha_4$ is uniformly externally connected to \[ \alpha_4' = \alpha_4 \begin{pmatrix}
 1 & 0 & 0 \\
 -a' & 1 & 0 \\
 0 & 0 & 1
\end{pmatrix}
=
\begin{pmatrix}
 * & * & * \\
 * & * & * \\
 0 & 1 & b''
\end{pmatrix}.
\]
Now we can find a polynomial $P$ of small degree such that if we first externally connect $\alpha_4'$ to \[\alpha_4'' = \alpha_4' \begin{pmatrix}
 1 & P & 0 \\
 0 & 1 & 0 \\
 0 & 0 & 1
\end{pmatrix},
\]
then the third column of \[\alpha_5 = \alpha_4'' \begin{pmatrix}
 1 & 0 & 0 \\
 0 & 1 & 1-b'' \\
 0 & 0 & 1
\end{pmatrix}
=
\begin{pmatrix}
 * & * & * \\
 * & * & * \\
 0 & 1 & 1
\end{pmatrix}
\]
remains $\varepsilon$-large.
\end{proof}

Now the vertex corresponding to $\alpha_5$ in the Cayley graph is adjacent to \[ \alpha_6 = \alpha_5 e_{2,3}(-1) = \begin{pmatrix}
 B & \begin{array}{l} * \\ * \end{array} \\
 \begin{array}{cc} 0 & 0 \end{array} & 1
\end{pmatrix},
\]
where $B \in \mathrm{SL}_2(\mathbb{F}_q[t])$.

\begin{claim4}
$\alpha_6$ is uniformly externally connected to \[ \alpha_7 = \begin{pmatrix}
 1 & 0 & x \\
 0 & 1 & y \\
 0 & 0 & 1
\end{pmatrix},
\]
which concludes the proof of Lemma \ref{lem1}.
\end{claim4}

\begin{proof}
Note that since the matrix $B^{-1}$ has determinant $1$, the entries of its first row are relatively prime. It follows from the Euclidean algorithm applied to these entries that $B^{-1}$ is a product of matrices \[ B^{-1} = \begin{pmatrix} 1 & 0 \\ q_1 & 1 \end{pmatrix} \begin{pmatrix} 0 & 1 \\ -1 & 0 \end{pmatrix}^{\pm 1} \begin{pmatrix} 1 & 0 \\ q_2 & 1 \end{pmatrix} \begin{pmatrix} 0 &  1 \\ - 1 & 0 \end{pmatrix}^{\pm 1} \cdots \begin{pmatrix} 1 & 0 \\ q_k & 1 \end{pmatrix}, \] where the polynomials $q_i$ are the quotients appearing when performing the Euclidean algorithm. In particular the quantity $\sum \deg q_i$ is bounded by the maximum of the degrees of the entries of $B^{-1}$. 

Now let us multiply $\alpha_6$ on the right by \[ \begin{pmatrix} 1 & 0 & 0 \\ q_1 & 1 & 0 \\ 0 & 0 & 1 \end{pmatrix} \begin{pmatrix} 0 & 1 & 0 \\ -1 & 0 & 0 \\ 0 & 0 & 1 \end{pmatrix}^{\pm 1} \begin{pmatrix} 1 & 0 & 0 \\ q_2 & 1 & 0 \\ 0 & 0 & 1 \end{pmatrix} \begin{pmatrix} 0 & 1 & 0 \\ -1 & 0 & 0\\ 0 & 0 & 1 \end{pmatrix}^{\pm 1} \cdots \begin{pmatrix} 1 & 0 & 0 \\ q_k & 1 & 0 \\ 0 & 0 & 1 \end{pmatrix}. \] Each product by a monomial matrix consists in moving to an adjacent vertex in the Cayley graph. Now the fact that $\sum \deg q_i$ is well controlled, together with Lemma \ref{tool} applied for each product by a matrix \[ \begin{pmatrix} 1 & 0 & 0 \\ q_i & 1 & 0 \\ 0 & 0 & 1 \end{pmatrix}, \] yield that $\alpha_6$ is uniformly externally connected to \[ \alpha_6 \begin{pmatrix}
 B^{-1} & \begin{array}{l} 0 \\ 0 \end{array} \\
 \begin{array}{cc} 0 & 0 \end{array} & 1
\end{pmatrix} = \begin{pmatrix}
 1 & 0 & x \\
 0 & 1 & y \\
 0 & 0 & 1
\end{pmatrix} = \alpha_7. \]
\end{proof}

\end{proof}

\section{Adding valuations} \label{secadd}

In this part we prove that the growth rate of the divergence function of the group $\mathrm{SL}_n(\mathcal{O}_{\mathcal{S}})$ cannot increase while adding valuations to the set $\mathcal{S}$.

\subsection{Number theory}
We now recall some basic definitions about valuations on a global function field $\Bbbk$.

\begin{defi}
A discrete valuation on $\Bbbk$ is a non-trivial homomorphism $\nu: \Bbbk^* \rightarrow \mathbb{R}$ satisfying $\nu(x+y) \geq \min(\nu(x),\nu(y))$ for all $x,y \in \Bbbk^*$ with $x+y \neq0$. It is convenient to extend $\nu$ to a function defined on $\Bbbk$ by setting $\nu(0)=\infty$.
\end{defi}

\begin{ex}
Let $P \in \mathbb{F}_q[t]$ be an irreducible polynomial. Every non-zero element $x$ of $\mathbb{F}_q(t)$ can be written in a unique way $x = P^n(a/b)$, where $n \in \mathbb{Z}$ and $a,b \in \mathbb{F}_q[t]$ are not divisible by $P$. We define a valuation on $\Bbbk = \mathbb{F}_q(t)$ putting $\nu_{P}(x) = n$. If $P = t-a$ then $\nu_{P}(x)$ is the order of vanishing of the rational function $x$ at the point $a \in \mathbb{F}_q$.
\end{ex}

\begin{ex} \label{exval}
Besides the valuations $\nu_{P}$ defined above, we get another discrete valuation $\nu_{\infty}$ on $\mathbb{F}_q(t)$ by setting $\nu_{\infty}(a/b) = \deg b - \deg a$ for any two non-zero polynomials $a,b \in \mathbb{F}_q[t]$.
\end{ex}

\begin{rmq}
Every discrete valuation on $\Bbbk = \mathbb{F}_q(t)$ is equivalent to either $\nu_{\infty}$ or $\nu_{P}$ for some irreducible $P \in \mathbb{F}_q[t]$ \cite[Th.2 Chap.III.1]{Weil}, where two valuations $\nu_1,\nu_2$ are said to be equivalent if there exists $c>0$ such that $\nu_1(x)=c \, \nu_2(x)$ for all $x \in \Bbbk$.
\end{rmq}

Given a valuation $\nu$ on $\Bbbk$, we define the associated norm by the formula $|x|_{\nu} = q^{-\nu(x)}$ for all $x \in \Bbbk$. It is easily checked that we get a metric on $\Bbbk$ by setting $d_{\nu}(x,y) = |x-y|_{\nu}$. By the completion of $\Bbbk$ with respect to the valuation $\nu$ we mean the completion of the metric space $(\Bbbk,d_{\nu})$. Note that the field operations and the valuation $\nu$ on $\Bbbk$ extend to its completion.

\begin{ex}
The completion of $\Bbbk = \mathbb{F}_q(t)$ with respect to $\nu_{P}$ is the field $\mathbb{F}_q(\!(P)\!)$, and its completion with respect to $\nu_{\infty}$ is $\mathbb{F}_q(\!(t^{-1})\!)$.
\end{ex}

If $\mathcal{S}$ denotes a finite set of valuations on $\Bbbk$, we denote by $\mathcal{O}_{\mathcal{S}}$ the ring of $\mathcal{S}$-integer points of $\Bbbk$. Recall that $\mathcal{O}_{\mathcal{S}}$ is defined as the set of $x \in \Bbbk$ such that $x$ is $\nu$-integral for all valuations $\nu \notin \mathcal{S}$,
\[ \mathcal{O}_{\mathcal{S}} = \left\{x \in \Bbbk \, : \, |x|_{\nu} \leq 1 \; \text{for all valuations} \; \nu \notin S \right\}.\]
We have a natural diagonal embedding of $\Bbbk$ into \[ \Bbbk_{\mathcal{S}} = \prod_{\nu \in \mathcal{S}} \Bbbk_{\nu},\]
where $\Bbbk_{\nu}$ denotes the completion of $\Bbbk$ with respect to the valuation $\nu$. Note that $\mathcal{O}_{\mathcal{S}}$ has a discrete and cocompact image into $\Bbbk_{\mathcal{S}}$.

\begin{ex}
Let $\Bbbk = \mathbb{F}_q(t)$ and let $\mathcal{S} = \left\{ \nu_{\infty}\right\}$. Then $\mathcal{O}_{\mathcal{S}}$ is the polynomial ring $\mathbb{F}_q[t]$. It is a discrete cocompact subring of $\mathbb{F}_q(\!(t^{-1})\!)$.
\end{ex}

\begin{ex}
Let $\Bbbk = \mathbb{F}_q(t)$ and let $\nu_a$ denote the valuation associated to the polynomial $t-a$. If $\mathcal{S} = \left\{\nu_{\infty}, \nu_0, \nu_{-1} \right\}$ then $\mathcal{O}_{\mathcal{S}} = \mathbb{F}_q[t,t^{-1},(t+1)^{-1}]$. It is a discrete cocompact subring of the locally compact ring $\mathbb{F}_q(\!(t^{-1})\!) \times \mathbb{F}_q(\!(t)\!) \times \mathbb{F}_q(\!(t+1)\!)$.
\end{ex}

\subsection{Proof of the result}

From now and until the end of the paper, $\mathcal{S}$ will denote a finite set of $s \geq 2$ pairwise non-equivalent valuations on $\Bbbk$ containing $\nu_{\infty}$. We will prove that the divergence function of the finitely generated group $\mathrm{SL}_n(\mathcal{O}_{\mathcal{S}})$ is linear. As in Section \ref{secmain}, we write down the arguments only for $n=3$, the proof of the general case being a straightforward extension of the proof for this case.

We now choose a finite generating set of $\mathrm{SL}_3(\mathcal{O}_{\mathcal{S}})$. We choose similarly the sets $S_0,S_1,S_2$ defined at the beginning of Section \ref{secmain}. Recall that according to Dirichlet unit theorem, the group $\mathcal{O}_{\mathcal{S}}^\times$ of units of $\mathcal{O}_{\mathcal{S}}$ is the direct product of a free Abelian group of rank $s-1$, freely generated by $\lambda_1, \ldots, \lambda_{s-1}$, with a finite Abelian group generated by $\lambda_s, \ldots, \lambda_d$. We define $S_3$ as the following set of matrices \[ S_3 = \bigcup_{i=1}^d \left\{
\begin{pmatrix} \lambda_i & & \\ & \lambda_i^{-1} & \\ & & 1 \end{pmatrix},
\begin{pmatrix} \lambda_i & & \\ & 1 & \\ & & \lambda_i^{-1} \end{pmatrix}
\right\}.
\]
We now take $S = S_0 \cup S_1 \cup S_2 \cup S_3$ as finite generating set of $\mathrm{SL}_3(\mathcal{O}_{\mathcal{S}})$.

If $x \in \mathcal{O}_{\mathcal{S}}$ and $\gamma$ is a matrix with entries in $\mathcal{O}_{\mathcal{S}}$, we denote by \[ |x| = \max_{\nu \in \mathcal{S}} |x|_{\nu} \] and \[ \left\| \gamma \right\| = \max_{i,j} \left\{ | \gamma_{i,j} | \right\}. \]

The control of the size of a matrix with respect to the word metric associated to $S$ by the size of its entries, provided by \cite{LMR}, still holds in this setting:

\begin{prop}
There exists $C > 0$ such that for any $\gamma \in \mathrm{SL}_3(\mathcal{O}_{\mathcal{S}})$, \[ C^{-1}(1 + \log \left\| \gamma \right\|) \leq d_{S}(e, \gamma) \leq C(1 + \log \left\| \gamma \right\|). \]
\end{prop}

We define similarly the notion of $\varepsilon$-large entry of an element of $\mathrm{SL}_3(\mathcal{O}_{\mathcal{S}})$, and the notion of external trajectories between two elements in the Cayley graph of $\mathrm{SL}_3(\mathcal{O}_{\mathcal{S}})$ associated to $S$. As in Section \ref{secmain} for $\mathrm{SL}_3(\mathbb{F}_q[t])$, we prove that any two elements of $\mathrm{SL}_3(\mathcal{O}_{\mathcal{S}})$ are uniformly externally connected. Some points of the proof will be similar to what we did in the case of $\mathbb{F}_q[t]$, but the idea is that here we are in a more pleasant situation because diagonal matrices coming from units in $\mathcal{O}_{\mathcal{S}}$ provide some more place to move.

\begin{lem}
There exist $c_1', c_2' > 0$ such that any $\alpha \in \mathrm{SL}_3(\mathcal{O}_{\mathcal{S}})$ can be uniformly externally connected to an element \[ \alpha' \in \begin{pmatrix} \mathrm{SL}_2(\mathcal{O}_{\mathcal{S}}) & \begin{array}{l} * \\ * \end{array} \\ \begin{array}{cc} 0 & 0 \end{array} & 1 \end{pmatrix} \] satisfying \[ c_1' \leq \frac{1+d_{S}(e,\alpha)}{1+d_{S}(e,\alpha')} \leq c_2'.\] 
\end{lem}

\begin{proof}
As before, the control on the size of $\alpha'$ will come from the fact that we proceed in a finite number of steps to connect $\alpha$ to $\alpha'$, and that each intermediate point satisfies such a control. Let $\varepsilon > 0$ be fixed. Write \[\alpha =
\begin{pmatrix}
 * & * & * \\
 * & * & * \\
 a & b & c
\end{pmatrix}. \]

First note that multiplying if necessary by diagonal elements of $S_3$, which has the effect of multiplying the columns of $\alpha$, we can assume that the first two columns belong to $\mathbb{F}_q[t]$ and that the third column is $\varepsilon$-large. Now proceeding as in Claim 1 of the proof of Lemma \ref{lem1} and using an analogue of Lemma \ref{tool}, we obtain that $\alpha$ is uniformly externally connected to \[\alpha_2 =
\begin{pmatrix}
 * & * & * \\
 * & * & * \\
 a' & b' & c
\end{pmatrix}, \]
where $a'$ and $b'$ are coprime. Then swapping columns 2 and 3, multiplying if necessary by elements of $S_3$ to get an $\varepsilon$-large third column, and doing as in Claim 2 of Section \ref{secmain}, we get that $\alpha_2$ is uniformly externally connected to \[\alpha_3 =
\begin{pmatrix}
 * & * & * \\
 * & * & * \\
 a' & 1 & b''
\end{pmatrix}. \]
Now using Lemma \ref{tool} we uniformly externally connect $\alpha_3$ to \[\alpha_4 =
\begin{pmatrix}
 * & * & * \\
 * & * & * \\
 0 & 1 & b''
\end{pmatrix}. \]
To obtain the desired result it is now sufficient to make the first column $\varepsilon$-large thanks to $S_3$, and use one more time Lemma \ref{tool} to connect $\alpha_4$ to \[ \begin{pmatrix}
 * & * & * \\
 * & * & * \\
 0 & 1 & 0
\end{pmatrix}, \]
and finally exchange columns 2 and 3.
\end{proof}

To finish the proof we now have to show that any two \[ \gamma, \gamma' \in \begin{pmatrix} \mathrm{SL}_2(\mathcal{O}_{\mathcal{S}}) & \begin{array}{l} * \\ * \end{array} \\ \begin{array}{cc} 0 & 0 \end{array} & 1 \end{pmatrix} \] are uniformly externally connected. But this is straightforward because since $s\geq2$, $\mathrm{SL}_2(\mathcal{O}_{\mathcal{S}})$ is finitely generated and quasi-isometrically embedded in $\mathrm{SL}_3(\mathcal{O}_{\mathcal{S}})$, so we easily get (by first making their third column large thanks to $S_3$) that each of them is uniformly externally connected to an element of the form \[ \begin{pmatrix}
 1 & 0 & * \\
 0 & 1 & * \\
 0 & 0 & 1
\end{pmatrix}. \]
Now in order to see that any two elements of this form are uniformly externally connected, we can for example argue that they lie in the quasi-isometrically embedded subgroup \[ \begin{pmatrix}
 * & 0 & * \\
 0 & * & * \\
 0 & 0 & *
\end{pmatrix} \simeq (\mathcal{O}_{\mathcal{S}} \rtimes \mathcal{O}_{\mathcal{S}}^\times)^2, \] which, as a direct product of two infinite finitely generated groups, trivially has a linear divergence.

\nocite{*}
\bibliographystyle{amsalpha}
\bibliography{SLd}

\providecommand{\bysame}{\leavevmode\hbox to3em{\hrulefill}\thinspace}
\providecommand{\MR}{\relax\ifhmode\unskip\space\fi MR }
\providecommand{\MRhref}[2]{%
  \href{http://www.ams.org/mathscinet-getitem?mr=#1}{#2}
}
\providecommand{\href}[2]{#2}
\begin{thebibliography}{LMR00}

\bibitem[Beh79]{Behr}
H.~Behr, \emph{{${\rm SL}_{3}({\mathbb{F}}_{q}[t])$} is not finitely
  presentable}, Homological group theory ({P}roc. {S}ympos., {D}urham, 1977),
  London Math. Soc. Lecture Note Ser., vol.~36, Cambridge Univ. Press,
  Cambridge, 1979, pp.~213--224.

\bibitem[Bro89]{Brown}
K.~Brown, \emph{Buildings}, Springer-Verlag, New York, 1989.

\bibitem[CDG10]{CDG}
P.E. Caprace, F.~Dahmani, and V.~Guirardel, \emph{Twin building lattices do not
  have asymptotic cut-points}, Geom. Dedicata \textbf{147} (2010), 409--415.

\bibitem[Cor08]{Co}
Y.~Cornulier, \emph{Dimension of asymptotic cones of {L}ie groups}, J. Topol.
  \textbf{1} (2008), no.~2, 342--361.

\bibitem[DMS10]{DMS}
C.~Dru{\c{t}}u, S.~Mozes, and M.~Sapir, \emph{Divergence in lattices in
  semisimple {L}ie groups and graphs of groups}, Trans. Amer. Math. Soc.
  \textbf{362} (2010), no.~5, 2451--2505.

\bibitem[Dru02]{Drutucone}
C.~Dru{\c{t}}u, \emph{Quasi-isometry invariants and asymptotic cones},
  Internat. J. Algebra Comput. \textbf{12} (2002), no.~1-2, 99--135,
  International Conference on Geometric and Combinatorial Methods in Group
  Theory and Semigroup Theory (Lincoln, NE, 2000).

\bibitem[DS05]{DS}
C.~Dru{\c{t}}u and M.~Sapir, \emph{Tree-graded spaces and asymptotic cones of
  groups}, Topology \textbf{44} (2005), no.~5, 959--1058, with an appendix by
  D. Osin and M. Sapir.

\bibitem[Gar97]{Garrett}
P.~Garrett, \emph{Buildings and classical groups}, Chapman \& Hall, London,
  1997.

\bibitem[Gro93]{Gro}
M.~Gromov, \emph{Asymptotic invariants of infinite groups}, Geometric group
  theory, {V}ol.\ 2 ({S}ussex, 1991), London Math. Soc. Lecture Note Ser., vol.
  182, Cambridge Univ. Press, Cambridge, 1993, pp.~1--295.

\bibitem[KL97]{KL}
B.~Kleiner and B.~Leeb, \emph{Rigidity of quasi-isometries for symmetric spaces
  and {E}uclidean buildings}, Inst. Hautes \'Etudes Sci. Publ. Math. (1997),
  no.~86, 115--197.

\bibitem[LMR93]{LMRbis}
A.~Lubotzky, S.~Mozes, and M.S. Raghunathan, \emph{Cyclic subgroups of
  exponential growth and metrics on discrete groups}, C. R. Acad. Sci. Paris
  S\'er. I Math. \textbf{317} (1993), no.~8, 735--740.

\bibitem[LMR00]{LMR}
\bysame, \emph{The word and {R}iemannian metrics on lattices of semisimple
  groups}, Inst. Hautes \'Etudes Sci. Publ. Math. (2000), no.~91, 5--53.

\bibitem[Mac56]{McD}
C.C. MacDuffee, \emph{The theory of matrices}, Chelsea Publishing Company, New
  York, 1956.

\bibitem[Nag59]{Nagao}
H.~Nagao, \emph{On {${\rm GL}(2, K[x])$}}, J. Inst. Polytech. Osaka City Univ.
  Ser. A \textbf{10} (1959), 117--121.

\bibitem[Sta68]{Stallings}
J.~R. Stallings, \emph{On torsion-free groups with infinitely many ends}, Ann.
  of Math. (2) \textbf{88} (1968), 312--334.

\bibitem[Wei95]{Weil}
A.~Weil, \emph{Basic number theory}, Classics in Mathematics, Springer-Verlag,
  Berlin, 1995.

\end{thebibliography}

\end{document}